\documentclass[10pt]{amsart}

\usepackage{graphicx, xcolor}
\usepackage{amssymb}
\usepackage{bbm}
\usepackage{amsmath,mathtools,enumerate}
\usepackage{amsthm}
\usepackage{stmaryrd}
\usepackage{tikz-cd}
\usepackage{faktor}
\usepackage{hyperref}

\newcommand{\C}{\mathbb{C}} 
\newcommand{\R}{\mathbb{R}} 
 
\newcommand{\Z}{\mathbb{Z}} 
\newcommand{\N}{\mathbb{N}} 

\renewcommand{\O}{\mathcal{O}}

\newcommand{\G}{\mathcal{G}}

\newcommand{\K}{\mathcal {K}}

\newcommand{\id}{\mathrm{id}}

\newcommand{\dom}{\mathrm{dom}}
\newcommand{\ran}{\mathrm{ran}}

\theoremstyle{definition}

\newtheorem{thm}{Theorem}[section]
\newtheorem{lem}[thm]{Lemma}
\newtheorem{prop}[thm]{Proposition}
\newtheorem{dfn}[thm]{Definition}

\newtheorem{rmk}[thm]{Remark}

\numberwithin{equation}{section}

\title[Principal Groupoid Models for Stable UCT Kirchberg Algebras]{Principal Groupoid Models for Stable UCT Kirchberg Algebras}

\author{Samuel Evington}
\address{Samuel Evington, Mathematical Institute, University of M\"unster, Ein\-stein\-strasse 62, 48149 M\"unster, Germany}
\email{evington@uni-muenster.de}

\author{Philipp Sibbel}
\address{Philipp Sibbel, Mathematical Institute, University of M\"unster, Ein\-stein\-strasse 62, 48149 M\"unster, Germany}
\email{philipp.sibbel@uni-muenster.de}

\thanks{Funded by the Deutsche Forschungsgemeinschaft (DFG, German Research Foundation) under Germany’s Excellence Strategy – EXC 2044/2 – 390685587, Mathematics Münster – Dynamics – Geometry – Structure; the Deutsche Forschungsgemeinschaft (DFG, German Research Foundation) – Project-ID 427320536 – SFB 1442; and ERC Advanced Grant 834267 - AMAREC}

\begin{document}
	
	\begin{abstract}
		We show that every stable UCT Kirchberg algebra has a principal étale groupoid model, and thus contains a C$^*$-diagonal. 
		Every unital UCT Kirchberg algebra $A$ for which $[1_A]_0$ has infinite order in $K_0(A)$ is also covered by our methods. In particular, we obtain a principal étale groupoid model for the Cuntz algebra $\mathcal{O}_\infty$. 
	\end{abstract}
	
	\maketitle
	
	\section*{Introduction}
	\numberwithin{equation}{section}
	\renewcommand{\thethm}{\Alph{thm}}
	
	Every simple, separable, amenable C$^*$-algebra that absorbs the Jiang--Su algebra $\mathcal{Z}$ tensorially (\cite{JS99}) and satisfies the universal coefficient theorem (UCT) of Rosenberg and Schochet (\cite{RS87}) is isomorphic to the reduced C$^*$-algebra of a twisted groupoid. However, it is an open question whether the twist can be eliminated and if the groupoid can be taken to be principal (i.e.\ the groupoid of an equivalence relation).
	
	In the stably finite case, the existence of a twisted groupoid model was established by Xin Li (\cite{li2020every}). His strategy was to use the classification machinery of the Elliott classification programme (see \cite{El95,El96, winter2018structure, WhiteICM, GL23, CGSTW} for an overview) to realise the C$^*$-algebra $A$ as an inductive limit of simpler building blocks and then build a Cartan subalgebra in $A$. By Kumjian--Renault theory, this leads to a twisted groupoid model for $A$ (\cite{kumjian1986c,renault2008cartan}). In this case, it is known that the groupoid can always be taken to be principal, but it is unknown if the twist can be eliminated in general.
	
	In the purely infinite case, the strategy has been to consider C$^*$-algebras associated to graphs and more general graph-like objects such as higher-rank graphs and topological graphs (\cite{spielberg2007graph,katsura2008class,yeend2006topological,li2019cartan, CFH20}). These C$^*$-algebras have natural groupoid models, which can be constructed by considering paths in the graph-like object.
	By the Kirchberg--Phillips theorem (\cite{Ki95,Ph00}), $K$-theory can be used to detect isomorphisms, so it suffices to exhaust all possible values for the $K$-theory. In this case, it is known that a twist is never needed, but it is unknown if the groupoid can be taken to be principal in general.
	
	In this paper, we prove that all  simple, separable, amenable, purely infinite C$^*$-algebras (i.e.\ Kirchberg algebras) that are stable and satisfy the UCT have a groupoid model that is both untwisted and principal.
	
	\begin{thm}\label{thm:stableKirchberg}
		Let $A$ be a stable Kirchberg algebra satisfying the UCT. Then there exists a principal, second-countable, locally compact, Hausdorff, étale groupoid $\G$ with a finite-dimensional unit space such that $A \cong C_r^*(\G)$. 
	\end{thm}
	
	Note that, by Kumjian--Renault theory (\cite{kumjian1986c,renault2008cartan}), it follows immediately from Theorem \ref{thm:stableKirchberg} that every stable UCT Kirchberg algebra contains a C$^*$-diagonal in the sense of Kumjian (\cite{kumjian1986c}) by considering the inclusion $C_0(\G^{(0)}) \subseteq  C_r^*(\G)$, where $\G^{(0)}$ denotes the unit space of $\G$. Our main motivation for seeking principal groupoid models is that Cartan subalgebras that are not C$^*$-diagonals automatically have infinite diagonal dimension (\cite{LLW23}) and therefore most likely lie outside the class of those groupoid models that one could reasonably hope to classify by a $K$-theoretic invariant.    
	
	Our work is inspired, at the conceptual level, by Spielberg's construction of groupoid models for stable UCT Kirchberg algebras (\cite{spielberg2007graph}) and at the technical level by Katsura's alternative construction using topological graph C$^*$-algebras (\cite{katsura2008class}). However, these existing constructions lead to groupoid models that are typically only topologically principal rather than principal. Under Kumjian--Renault theory, the resulting inclusions $C_0(\G^{(0)}) \subseteq  C_r^*(\G)$ are Cartan subalgebras but need not be C$^*$-diagonals (see Sections \ref{subsec:groupoids} and \ref{subsec:cartans} for the relevant definitions).
	
	The key idea in our construction is to eliminate the isotropy that arises in Katsura's approach due to cycles in the topological graph (\cite{katsura2008class}). This is achieved by integrating a minimal topological dynamical system into the construction. This approach has its roots in the work of Winter and the authors on C$^*$-diagonals in the Cuntz algebra $\O_n$ for $2 \leq n < \infty$ (see \cite{sibbel2024cantor,ES26,ES26preprint}). 
	
	One important novelty of this paper is the jump from considering Cantor minimal systems to more general topological dynamical systems. This has proven necessary for $K$-theoretic reasons. Indeed, using a Cantor minimal system to eliminate isotropy typically changes the $K$-theory of the groupoid C$^*$-algebra. In \cite{ES26}, the authors were able to correct for this at a later stage of the construction, but this is only viable in certain special cases.
	
	The topological dynamical system that we make use of is due to Deeley--Putnam--Strung (\cite{deeley2018constructing}), who constructed a minimal homeomorphism $\rho$ on an infinite, compact, metric space $Z$ that has the same $K$-theory as a point. The space $Z$ is constructed using the topological properties of odd-dimensional spheres (see Section \ref{subsec:point-like} for an overview) and was used in the construction of groupoid models for the Jiang--Su algebra $\mathcal{Z}$.   
	
	The output of our construction is summarised by the following theorem. 
	\begin{thm}\label{thm:general}
		Let $X$ be a second-countable, locally compact, Hausdorff space. Then there exists a principal, second-countable, locally compact, Hausdorff, étale groupoid $\G$ such that $C_r^*(\G)$ is a UCT Kirchberg algebra satisfying
		\begin{align*}
			K_0(C_r^*(\G)) &\cong K_0(C_0(X)),\\
			K_1(C_r^*(\G)) &\cong K_1(C_0(X)).
		\end{align*}
		If $X$ has finite covering dimension, then the unit space $\G^{(0)}$ has finite covering dimension.
		If $X$ is compact, then the unit space $\G^{(0)}$ is compact, the C$^*$-algebra $C_r^*(\G)$ is unital, and the isomorphism of $K_0$-groups can be chosen to preserve the $K_0$-class of the unit. 
	\end{thm}
	Theorem~\ref{thm:stableKirchberg} follows from Theorem~\ref{thm:general} by selecting the space $X$ to have the same $K$-theory as the stable UCT Kirchberg algebra $A$, stabilising the groupoid $\G$ (to ensure the corresponding C$^*$-algebra is also stable), and then applying the Kirchberg--Phillips theorem (\cite{Ki95,Ph00}) to show that $A \cong C_r^*(\G)$. 
	Theorem~\ref{thm:general} can also be applied to obtain principal groupoid models for a large class of unital UCT Kirchberg algebras. 
	
	\begin{thm}\label{thm:MostUnitalKirchberg}
		Let $A$ be a unital Kirchberg algebra satisfying the UCT. Suppose $[1_A]_0$ has infinite order and generates a direct summand of $K_0(A)$. Then there exists a principal, second-countable, locally compact, Hausdorff, étale groupoid $\G$ with a finite-dimensional unit space such that $A \cong C_r^*(\G)$. In particular, $A$ contains a C$^*$-diagonal.
	\end{thm}
	The $K$-theoretic restriction in Theorem~\ref{thm:MostUnitalKirchberg} is necessary and sufficient for there to be a compact space $X$ such that $(K_0(A), [1_A]_0) \cong (K_0(C(X)), [1_{C(X)}]_0)$ and $K_1(A) \cong K_1(C(X))$. 
	We obtain Theorem~\ref{thm:MostUnitalKirchberg} by applying Theorem~\ref{thm:general} with this space $X$ and using the unital case of the Kirchberg--Phillips theorem. 
	
	As a special case of Theorem~\ref{thm:MostUnitalKirchberg}, we obtain the first example of a principal groupoid model for the Cuntz algebra $\mathcal{O}_\infty$. 
	In this case, the compact space $X$ used in the construction is just the one-point space, and we are able to identify the unit space of the resulting groupoid.
	
	\begin{thm}\label{thm:O_infinity}
		There exists a principal, second-countable, locally compact, Hausdorff, étale groupoid $\G_\infty$ such that 
        $C_r^*(\G_\infty) \cong \O_\infty$. 
		In particular, $\O_\infty$ contains a C$^*$-diagonal. Moreover, the unit space of $\G_\infty$ is the product of a Cantor space and a Deeley--Putnam--Strung space and has dimension at most three.
	\end{thm}
	
	This positive result for $\O_\infty$ is of particular interest as neither the construction of Brown--Clark--Sierakowski--Sims  (\cite{brown2016purely}) nor the construction of principal groupoid models for the Cuntz algebras $\O_n$ for $2 \leq n < \infty$ due to Winter and the authors (\cite{sibbel2024cantor,ES26}) can handle the case of $\O_\infty$. Combining \cite[Theorem A]{ES26} with Theorem~\ref{thm:O_infinity}, every Cuntz algebra contains a C$^*$-diagonal.
	
	The existence of a principal groupoid model for $\O_\infty$ can be leveraged to construct additional principal groupoid models for Kirchberg algebras. Indeed, starting with some principal groupoid model $\G$ for the Kirchberg algebra $A$ with a finite-dimensional unit space $\G^{(0)}$, one can obtain a countably infinite family of principal groupoid models by taking products of $\G$ with finitely many copies of the groupoid $\G_\infty$ since $A \cong A \otimes \O_\infty$ and $\O_\infty \cong \O_\infty \otimes \O_\infty$. The groupoids in such a family can be distinguished by considering the dimension of their unit spaces. (As experts will be aware, a similar trick can also be performed using a principal groupoid model of the Jiang--Su algebra $\mathcal{Z}$ instead of $\O_\infty$; see for example \cite[Proposition 5.1]{li2019cartan}.)
	
	We can also consider the product of $\G_\infty$ with a (twisted) principal groupoid model $\G_B$ of a stably finite C$^*$-algebra $B$. The result is a (twisted) principal groupoid model for the purely infinite C$^*$-algebra $B \otimes \O_\infty$. Applying this observation to the results of Xin Li on (twisted) principal groupoid models for stably finite, unital C$^*$-algebras covered by the Elliott classification programme (\cite{li2020every}), we obtain the following result. 
	\begin{thm}\label{thm:EvenMoreUnitalKirchberg}
		Let $A$ be a unital Kirchberg algebra satisfying the UCT. Suppose $[1_A]_0$ has infinite order in $K_0(A)$. Then there exists a principal, second-countable, locally compact, Hausdorff, étale groupoid $\G$ with a finite-dimensional unit space and a twist $\Sigma$ such that $A \cong C_r^*(\G, \Sigma)$. In particular, $A$ contains a C$^*$-diagonal.
	\end{thm}
	The $K$-theoretic restriction on $A$ in Theorem~\ref{thm:EvenMoreUnitalKirchberg} is weaker than in Theorem~\ref{thm:MostUnitalKirchberg}; however, unlike Theorem~\ref{thm:MostUnitalKirchberg}, the construction of Theorem~\ref{thm:EvenMoreUnitalKirchberg} results in a twisted groupoid whenever Xin Li's construction requires one. Note that, if $B$ is unital, then stable finiteness implies that $[1_B]_0$ must have infinite order in $K_0(B)$, so $[1_B \otimes 1_{\O_\infty}]_0$ must have infinite order in $K_0(B \otimes \O_\infty)$. Therefore, the $K$-theoretic restriction in Theorem~\ref{thm:EvenMoreUnitalKirchberg} is inherent to this strategy.
	
	In light of Theorem~\ref{thm:MostUnitalKirchberg}, Theorem~\ref{thm:EvenMoreUnitalKirchberg}, and the existence of principal groupoid models for all Cuntz algebras (by \cite{ES26} and Theorem~\ref{thm:O_infinity}), we expect that every unital UCT Kirchberg algebra has a principal groupoid model. 
	
	\section{Preliminaries}
	\renewcommand{\thethm}{\arabic{thm}}
	\numberwithin{thm}{section}
	
	\subsection{Kirchberg algebras}
	A simple C$^*$-algebra $A$ is said to be \emph{purely infinite} if $A \not\cong \C$ and for any $a,b \in A_+$ with $b \neq 0$ there exists $r \in A$ such that $a = r^*br$. A \emph{Kirchberg algebra} is a simple, separable, amenable, purely infinite C$^*$-algebra. The most famous examples of Kirchberg algebras are the Cuntz algebras $\O_n$ for $2 \leq n \leq \infty$ (\cite{cuntz1977simple}). For $n$ finite, $\O_n$ is defined as the universal C$^*$-algebra generated by $n$ isometries $s_1,\ldots,s_n$ with $\sum_{i=1}^n s_is_i^* = 1$, and $\O_\infty$ is defined as the universal C$^*$-algebra generated by a sequence of isometries $(s_i)_{i=1}^\infty$ with pairwise orthogonal ranges.
	
	Kirchberg proved that a simple, separable, amenable C$^*$-algebra $A$ is purely infinite if and only if $A \cong A \otimes \O_\infty$ (\cite{Ki95}). Moreover, by the Kirchberg--Phillips theorem (\cite{Ki95,Ph00}), Kirchberg algebras that satisfy the UCT are classified by their $K$-theory. More precisely, two unital UCT Kirchberg algebras $A$ and $B$ are isomorphic if and only if $((K_0(A),[1_A]_0), K_1(A)) \cong ((K_0(B),[1_B]_0), K_1(B))$, and stable UCT Kirchberg algebras $A$ and $B$ are isomorphic if and only if $(K_0(A), K_1(A)) \cong (K_0(B), K_1(B))$; see for example \cite[Theorem 8.4.1]{rordam2002classification}. These two results cover all cases as Kirchberg algebras are either unital or stable (\cite{Zhang92}). It is also known that any pair of countable abelian groups can be realised as the $K$-theory of a UCT Kirchberg algebra $A$ and, in the unital case, any group element can be realised as $[1_A]_0$; see for example \cite[Section 4.3]{rordam2002classification}.
	
	\subsection{Topological groupoids and their C*-algebras}\label{subsec:groupoids}
	
	We assume the reader is familiar with the basic theory of topological groupoids and their associated C$^*$-algebras; see for example \cite{renault1980groupoid, renault2008cartan}. The main purpose of this subsection is to fix our notation.
	
	In this paper, all groupoids considered will be locally compact and Hausdorff.
	We write $\G^{(0)}$ for the unit space of a topological groupoid $\G$, and we denote the range and source maps by $r,s: \G \to \G^{(0)}$, respectively. We recall that $\G$ is said to be \emph{étale} if $r$ and $s$ are local homeomorphisms.
	
	For $x\in \G^{(0)}$, the subgroup $\G_x^x = \{g\in \G : r(g)=s(g)=x \}$ is called the \emph{isotropy subgroup at $x$}. 
	The groupoid $\G$ is said to be \emph{principal} if $\G_x^x=\{x\}$ for all $x\in \G^{(0)}$ and is said to be \emph{topologically principal} if the set of all $x\in \G^{(0)}$ with $\G_x^x=\{x\}$ is dense in $\G^{(0)}$.
	
	We write $C_r^*(\G,\Sigma)$ for the reduced C$^*$-algebra of the locally compact, Hausdorff, étale groupoid $\G$ with respect to the twist $\Sigma$; see \cite[Section 4]{renault2008cartan} for details of the construction. The C$^*$-algebra $C_r^*(\G,\Sigma)$ contains a canonical abelian subalgebra isomorphic to $C_0(\G^{(0)})$.
	
	\subsection{Kumjian--Renault theory}\label{subsec:cartans}
	
	Cartan subalgebras and C$^*$-diagonals are special classes of abelian subalgebras with close connections to groupoid C$^*$-algebras thanks to the work of Kumjian (\cite{kumjian1986c}) and Renault (\cite{renault2008cartan}).  We first recall the relevant definitions. 
	
	\begin{dfn}[{\cite[Definition 5.1]{renault2008cartan}}]
		Let $A$ be a C$^*$-algebra. A non-degenerate sub-C$^*$-algebra $D \subseteq A$ is called a \emph{Cartan subalgebra} if 
		\begin{enumerate}
			\item[(i)] $D$ is a maximal abelian sub-C$^*$-algebra of $A$,
			\item[(ii)] $A$ is generated by $\mathcal{N}_A(D) = \{n \in A : n^*Dn \subseteq D \text{ and } nDn^* \subseteq D\}$, and
			\item[(iii)] there exists a faithful conditional expectation $A \rightarrow D$.
		\end{enumerate}
		Moreover, $D \subseteq A$ is called a C$^*$-\textit{diagonal} if additionally
		\begin{enumerate}
			\item[(iv)] every pure state on $D$ has a unique extension to a pure state on $A$.
		\end{enumerate}
	\end{dfn}
	
	We now provide a formal statement of Kumjian--Renault theory.
	
	\begin{thm}[{\cite{kumjian1986c,renault2008cartan}}] \label{thm:renault-cartan} There is a one-to-one correspondence between Cartan subalgebras and twisted étale groupoids. More precisely: 
		\begin{enumerate}[(i)]
			\item Let $\G$ be a topologically principal, second-countable, locally compact, Hausdorff, étale groupoid and let $\Sigma$ be a twist over $\G$. Then $C_r^*(\G,\Sigma)$ is separable and $C_0(\G^{(0)}) \subseteq C_r^*(\G,\Sigma)$ is a Cartan subalgebra. Moreover, $C_0(\G^{(0)}) \subseteq C_r^*(\G,\Sigma)$ is a C$^*$-diagonal if and only if $\G$ is principal.
			\item Let $A$ be a separable C$^*$-algebra and $D \subseteq A$ be a Cartan subalgebra. Then there is a topologically principal, second-countable, locally compact, Hausdorff, étale groupoid $\G$ and a twist $\Sigma$ such that the inclusion $D \subseteq A$ is isomorphic to $C_0(\G^{(0)}) \subseteq C_r^*(\G,\Sigma)$. Moreover, the twisted groupoid $(\G,\Sigma)$ is uniquely determined up to isomorphism by the inclusion $D \subseteq A$.  
		\end{enumerate}
	\end{thm} 
	
	We will make frequent use of Kumjian--Renault theory to reformulate groupoid results in the language of Cartan subalgebras and vice versa.
	
	\subsection{Topological graphs}\label{subsec:topological-graphs}
	In this section, we recall the definition and basic properties of topological graphs and their associated C$^*$-algebras.  Our main references are the foundational papers of Katsura on topological graphs (\cite{katsura2004class,katsura2006class2,katsura2006class3,katsura2008class}).
	
	We begin with the definition of a topological graph.
 	
	\begin{dfn}(\cite[Definition 2.1]{katsura2004class})\label{def:top-graph}
		A \emph{topological graph} $E= (E^0, E^1,d,r)$ consists of two locally compact Hausdorff spaces $E^0$ and $E^1$ together with two maps $d,r: E^1 \to E^0$, where $d$ is a local homeomorphism and $r$ is continuous. 
	\end{dfn}
	
	Next, we set up the notation needed to define the C$^*$-algebra of a topological graph (see \cite{katsura2004class,katsura2008class}).
	Let $E = (E^0, E^1,d,r)$ be a topological graph. 
	Set 
	\begin{equation}
		C_d(E^1) = \{\xi \in C(E^1) : \langle \xi, \xi \rangle \in C_0(E^0)\},
	\end{equation}
	where the inner product is defined by 
	\begin{equation}
		\langle \xi, \eta \rangle (x) = \sum_{e \in d^{-1}(x)} \overline{\xi(e)} \eta(e)
	\end{equation}
	for $\xi, \eta \in C_d(E^1)$ and $x \in E^0$. 
	We define a left and a right action of $C_0(E^0)$ on $C_d(E^1)$ by
	\begin{equation}
		(f \xi g) (e) = f(r(e)) \xi(e) g(d(e))
	\end{equation}
	for $f,g \in C_0(E^0), \xi \in C_d(E^1)$, and $e \in E^1$.  
	The left action can be viewed as a $^*$-homomorphism $\pi_r: C_0(E^0) \rightarrow \mathcal{L}(C_d(E^1))$ into the C$^*$-algebra of all adjointable operators on the Hilbert $C_0(E^0)$-module $C_d(E^1)$. 
	
	A vertex $v \in E^0$ is said to be \emph{regular} if there exists a neighbourhood $V$ of $v$ such that $r^{-1}(V) \subseteq E^1$ is compact, and $r(r^{-1}(V)) = V$. The set of all regular vertices forms an open subset $E^0_{rg} \subseteq E^0$. The complement $E_{sg}^0 = E^0 \setminus E_{rg}^0$ is closed and its elements are called \emph{singular} vertices. 
	
	Given $\xi,\eta \in C_d(E^1)$, we write $\theta_{\xi,\eta} \in \mathcal{L}(C_d(E^1))$ for the rank-one operator defined by $\theta_{\xi,\eta} (\zeta) = \xi \langle \eta, \zeta \rangle$ for all $\zeta \in C_d(E^1)$. We recall that the closed linear span of the rank-one operators forms an ideal $\mathcal{K}(C_d(E^1))$ of $\mathcal{L}(C_d(E^1))$. The restriction of the $^*$-homomorphism $\pi_r:C_0(E^0) \rightarrow \mathcal{L}(C_d(E^1))$ to $C_0(E^0_{rg}) \subseteq C_0(E^0)$ is an injective map $C_0(E^0_{rg}) \rightarrow \mathcal{K}(C_d(E^1))$.
	
	We are now in a position to recall the definition of topological graph C$^*$-algebras.
	
	\begin{dfn}\label{dfn:TopologicalGraphAlgebra} (\cite[Definition 1.3]{katsura2008class})
		For a topological graph $E= (E^0, E^1,d,r)$, the C$^*$-algebra $\mathcal{O}(E)$ is the universal C$^*$-algebra generated by the images of a $^*$-homomorphism $t^0:C_0(E^0) \to \mathcal{O}(E)$ and a linear map $t^1:C_d(E^1) \to \mathcal{O}(E)$ satisfying 
		\begin{enumerate}
			\item[(i)]  $t^1(\xi)^*t^1(\eta) = t^0(\langle \xi, \eta \rangle) $ for $\xi, \eta \in C_d(E^1)$,
			\item[(ii)] $t^0(f)t^1(\xi) = t^1 (\pi_r(f)\xi)$ for $f \in C_0(E^0)$ and $\xi \in C_d(E^1)$,
			\item[(iii)] $t^0(f) = \varphi (\pi_r(f))$ for $f \in C_0(E^0_{rg})$,
		\end{enumerate}
		where $\varphi: \K(C_d(E^1)) \to \mathcal{O}(E)$ is the $^*$-homomorphism defined on rank-one operators by $\varphi(\theta_{\xi,\eta}) = t^1(\xi) t^1(\eta)^*$ for $\xi, \eta \in C_d(E^1)$, which is well-defined by (i).
	\end{dfn}
	
	In order to discuss Katsura's sufficient conditions for a topological graph C$^*$-algebra to be a UCT Kirchberg algebra, we shall need to introduce some notation for paths in a topological graph $E= (E^0, E^1,d,r)$.
	
	For $n \geq 2$, we define the space of paths of length $n$ by
	\begin{equation}
		E^n = \{(e_1,\ldots,e_n) \in (E^1)^n: d(e_i) = r(e_{i+1}) \mbox { for } i=1,\ldots,n-1\}.
	\end{equation}
	We write $E^* = \bigsqcup_{n=0}^\infty E^n$ for the space of all finite paths (including paths of length zero and one).
	We also define the space of infinite paths
	\begin{equation}
		E^\infty = \{(e_i)_{i=1}^\infty \in (E^1)^\N: d(e_i) = r(e_{i+1}) \text{ for } i \in \N\}.
	\end{equation}
	
	The domain and range maps are defined on $E^n$ by $d((e_1,\ldots,e_n)) = d(e_n)$ and $r((e_1,\ldots,e_n)) = r(e_1)$, for $n \geq 2$, and taken to be the identity on $E^0$. The resulting maps $d,r:E^* \rightarrow E^0$ are easily seen to be continuous. Similarly, we define a continuous map $r:E^\infty \rightarrow E^0$ by $r((e_i)_{i=1}^\infty) = r(e_1)$ for all $(e_i)_{i=1}^\infty \in E^\infty$. 
	For a finite or infinite path $\mu$, we denote its length by $|\mu| \in \N_0 \cup \{\infty\}$ and its constituent edges by $\mu_i$ for $i \in \N$ with  $1 \leq i \leq |\mu|$.
	
	We define the \emph{positive orbit space} $\text{Orb}^+(v)$ of $v \in E^0$ by 
	\begin{equation}
		\text{Orb}^+(v) = \{ r(\mu) \in E^0 :  \mu \in E^*, d(\mu)= v \}.
	\end{equation}
	For $v \in E^0$, a \emph{negative orbit} of $v$ is either a finite path $\mu \in E^*$ with $r(\mu)= v$ and $d(\mu) \in E^0_{sg}$, or an infinite path $\mu \in E^\infty$ with $r(\mu )= v$.
	Given $v \in E^0$ and a negative orbit $\mu$ of $v$, the \emph{negative orbit space} $\text{Orb}^-(v,\mu)$ is defined by
	\begin{equation}
		\text{Orb}^-(v,\mu) = \{v\} \cup \{d(\mu_i): i \in \N, 1 \leq i \leq |\mu| \}.
	\end{equation}
	The \emph{orbit space} $\text{Orb}(v,\mu)$ of $v$ with respect to the negative orbit $\mu$ of $v$ then is defined by
	\begin{equation}
		\text{Orb}(v,\mu) = \bigcup_{v^\prime \in \text{Orb}^-(v,\mu)} \text{Orb}^+(v^\prime).
	\end{equation}

	We are now in a position to recall Katsura's notion of \emph{minimality} of a topological graph, which is closely related to the simplicity of the associated C$^*$-algebra.
	
	\begin{dfn}[{\cite[Definition 1.8]{katsura2008class}}]
		A topological graph $E= (E^0,E^1,d,r)$ is said to be \emph{minimal} if $\text{Orb}(v,\mu)$ is dense in $E^0$ for every $v \in E^0$ and every negative orbit $\mu$ of $v$.
	\end{dfn}
	
	Next, we shall review Katsura's notion of a \emph{contracting} topological graph, which is closely related to pure infiniteness of the associated C$^*$-algebra. 
	For this, we introduce some notation.
	Let $n, m \in \N$ with $k= \min\{n,m\}$, $U \subseteq E^n$ and $U^\prime \subseteq E^m$. We define $U \pitchfork U^\prime \subseteq E^k$ by 
	\begin{equation}
		U \pitchfork U^\prime  = (U\vert_k) \cap (U^\prime\vert_k),
	\end{equation} 
	where $U\vert_k = \{(\mu_1,\mu_2,\dots,\mu_k): \mu \in U\}$ and $U^\prime \vert_k = \{(\mu_1,\mu_2,\dots,\mu_k): \mu \in U^\prime\}$.

	\begin{dfn}[{\cite[Definitions 2.3 and 2.7]{katsura2008class}}]
		Let $E= (E^0,E^1,d,r)$ be a topological graph. We say that a non-empty open subset $V$ of $E^0$ is a \textit{contracting} open set if the closure $\overline{V}$ is compact and there exist $N \in \N$ and non-empty open subsets $U_k \subseteq E^{n_k}$ for $k \in \{ 1,\dots,N\}$, where $1 \leq n_k < \infty $, such that
		\begin{enumerate}[(i)]
			\item $r (U_k) \subseteq V $ for $k \in \{ 1,\dots,N\}$;
			\item $U_k \pitchfork U_l = \emptyset $  for $k,l \in \{ 1,\dots,N\}$ with $k \neq l$;
			\item $\overline{V} \subsetneq \bigcup_{k=1}^N d (U_k)$.  
		\end{enumerate}
		A topological graph $E= (E^0, E^1,d,r)$ is called \textit{contracting at $v_0 \in E^0$} if $\overline{\text{Orb}^+(v_0)}= E^0$ and any neighbourhood $V_0$ of $v_0$ contains a contracting open set $V \subseteq V_0$. We call $E$ \textit{contracting} if $E$ is contracting at some $v_0 \in E^0$. 
	\end{dfn}
	
	Finally, we recall that a topological graph $(E^0,E^1,d,r)$ is \emph{second-countable} if both $E^0$ and $E^1$ are second-countable. We are now able to state Katsura's result about when topological graph C$^*$-algebras are UCT Kirchberg algebras. (Note that Katsura's definition of Kirchberg algebras in \cite{katsura2008class} includes the UCT.)
	
	\begin{thm} (\cite[Corollary B]{katsura2008class}) \label{thm:graph-contracting-Kirchberg-algebra}
		Let $E$ be a second-countable, minimal, contracting topological graph. Then $\mathcal{O}(E)$ is a UCT Kirchberg algebra. 
	\end{thm}
	
	To compute the $K$-theory of the topological graph C$^*$-algebras that we construct in the proof of Theorem \ref{thm:general}, we shall make use of Katsura's 6-term exact sequence.
	
	\begin{prop} (\cite[Corollary 6.10]{katsura2004class})\label{prop:graph-k-theory}
		Suppose $E= (E^0, E^1,d,r)$ is a second-countable topological graph. Then there exists a 6-term exact sequence of the form 
		
		\begin{equation}\label{eqn:PV-6-term-exact-seq}
			\begin{tikzcd}[column sep=1cm]
				K_0(C_0(E_{rg}^0)) \arrow[r]  & K_0(C_0(E^0)) \arrow[r, "t^0_*"] & K_0(\mathcal{O}(E))\arrow[d]\\
				K_1(\mathcal{O}(E)) \arrow[u] & K_1(C_0(E^0)) \arrow[l, "t^0_*"]  & K_1(C_0(E_{rg}^0)) \arrow[l],
			\end{tikzcd}
		\end{equation} 
		where $t^0:C_0(E^0) \rightarrow \mathcal{O}(E)$ is the canonical $^*$-homomorphism (see Definition \ref{dfn:TopologicalGraphAlgebra}) and $t^0_*$ denotes the induced map on $K$-theory. Furthermore, if $E_{rg}^0 = \emptyset$, then $t^0$ induces isomorphisms  
		$K_0(C_0(E^0)) \rightarrow K_0(\mathcal{O}(E))$ and 
		$K_1(C_0(E^0)) \rightarrow K_1(\mathcal{O}(E))$.
	\end{prop}
	\begin{proof}
		This is a restatement of \cite[Corollary 6.10]{katsura2004class}, where the remaining maps in the 6-term exact sequence are defined. When $E_{rg}^0 = \emptyset$, the 6-term exact sequence is still valid with $C_0(E_{rg}^0)$ interpreted as the zero C$^*$-algebra. Exactness then ensures that $t^0$ induces isomorphisms on $K$-theory.
	\end{proof}
	
	\subsection{Deaconu--Renault groupoids}\label{subsec:DR-groupoids}
	In this section, we recall the definiton of an important class of examples of groupoids, introduced by Deaconu (\cite{deaconu1995groupoids}) and Renault (\cite{renault2000cuntz}). We will use \cite{renault2000cuntz} as a reference.
	
	A \emph{partial local homeomorphism} $\sigma$ on a topological space $X$ is a local homeomorphism from one open subset $\dom(\sigma) \subseteq X$ to another open subset $\ran(\sigma) \subseteq X$. The composition $\sigma_1 \circ \sigma_2$ of two partial local homeomorphisms $\sigma_1$ and $\sigma_2$ of $X$ is a well-defined partial local homeomorphism on the open set $\sigma_2^{-1}(\dom(\sigma_1) \cap \ran(\sigma_2))$, although this open set may be empty in general. We write $\sigma^k$ for the $k$-fold composition of a single partial local homeomorphism $\sigma$ on $X$, where by convention $\sigma^0 = \id_X$ and $\sigma^1 = \sigma$.
	
	The \emph{Deaconu--Renault groupoid} $\G(X,\sigma)$ associated to a partial local homeomorphism $\sigma$ on a second-countable, locally compact, Hausdorff space $X$ is defined as follows. The underlying set of  $\G(X,\sigma)$ is  
	\begin{equation}
		\{(x,n - m, y): n,m \in \N_0, x \in \dom(\sigma^n), y \in \dom(\sigma^m), \sigma^n(x) = \sigma^m(y)\}.
	\end{equation}
	The unit space is $\{(x,0,x): x \in X\}$, which is identified with $X$ via $x \mapsto (x,0,x)$. The range and source maps are given by $r(x,k, y) = x$ and $s(x,k, y) = y$. Multiplication is given by $(x,k,y)(y,\ell,z) = (x,k+\ell,z)$, and the inverse is given by $(x,k,y)^{-1} = (y,-k,x)$. The topology on $\G(X,\sigma)$ is generated by the basic open sets $B(U,n,m,V)$ given by
	\begin{equation}
		B(U,n,m,V) = \{(x,n - m, y): x \in U, y \in V, \sigma^n(x) = \sigma^m(y)\}, 
	\end{equation}
	where $n,m \in \N_0$, $U$ is an open subset of $\dom(\sigma^n)$, $V$ is an open subset of $\dom(\sigma^m)$, and both $\sigma^n\vert_{U}$ and $\sigma^m\vert_{V}$ are injective.
	One can then verify that $\G(X,\sigma)$ is a second-countable, locally compact, Hausdorff, étale groupoid.
	
	\subsection{Groupoids associated to topological graphs}\label{subsec:top-graphs-groupoids}
	
	The first groupoid models for topological graph C$^*$-algebras were constructed by Yeend (\cite{yeend2006topological,Yeend07}). This construction was further studied and extended by Kumjian and Li (\cite{kumjian2017twisted}). In this subsection, we review the construction of groupoids associated to topological graphs, following the conventions and notation of \cite{kumjian2017twisted}.
	
	The \emph{boundary path space} $\partial E$ of a topological graph $E=(E^0, E^1,d,r)$ is the set 
	\begin{equation}
		\partial E =  E^\infty \sqcup \{\mu \in E^* : d(\mu) \in E^0_{sg}\},
	\end{equation}
	equipped with a natural locally compact, Hausdorff topology. This topology is defined as follows:
	For a subset $S \subseteq E^*$, we set
	\begin{equation}
		Z(S) = \{ \mu \in \partial E : r(\mu) \in S \text{ or }  (\mu_1, \ldots, \mu_i) \in S \text{ for some } i \le |\mu|\}.
	\end{equation}
	Then $\partial E$ carries the topology generated by the basic open sets $Z(U) \cap Z(K)^c$,
	where $U$ is an open subset of $E^*$ and $K$ is a compact subset of $E^*$ (see \cite[Definition 3.7]{kumjian2017twisted}). Assuming $E$ is second-countable, the topology on $\partial E$ is second-countable, and so is completely determined by the convergence of sequences and thus by the following lemma of Kumjian--Li.
	
	\begin{lem}[{\cite[Lemma 3.8]{kumjian2017twisted}}]\label{lem:Kumjian-Li}
		Let $E$ be a topological graph. Fix a sequence $(\mu^{(n)})_{n=1}^\infty \subseteq \partial E$, and fix $\mu \in \partial E$. Then $\mu^{(n)} \to \mu$ if and only if
		\begin{enumerate}[(i)]
			\item $r(\mu^{(n)}) \to r(\mu)$;
			\item for $1 \le i \le |\mu|$ with $i \ne \infty$, there exists $N \ge 1$ such that $|\mu^{(n)}| \ge i$ whenever $n \ge N$ and
			$(\mu^{(n)}_1,\ldots,\mu^{(n)}_i) \to (\mu_1, \ldots, \mu_i)$ as $n \to \infty$; 
			\item if $|\mu| < \infty$, then the set $\{ n : |\mu^{(n)}| > |\mu| \text{ and } \mu^{(n)}_{|\mu|+1} \in K \}$ is finite for any compact set $K \subseteq E^1$.
		\end{enumerate}
	\end{lem}
	
	There is a natural partial local homeomorphism $\sigma_E$ on $\partial E$ with domain $\partial E \setminus E^0_{sg}$ defined by combining the left shift $E^\infty \rightarrow E^\infty$, the left shift  $E^n \rightarrow E^{n-1}$ for all $n \ge 2$, and the domain map $d$ on $E^1$ (see \cite[Lemma 6.1]{kumjian2017twisted}). We can now state the main result of this subsection.
	
	\begin{prop}[{\cite[Theorem 5.2]{yeend2006topological}, \cite[Theorem 6.7]{kumjian2017twisted}}]\label{prop:groupoid-model-for-graphs} 
		Let $E= (E^0, E^1,d,r)$ be a second-countable topological graph. 
        Let $\G(\partial E,\sigma_E)$ be the Deaconu--Renault groupoid associated to the left shift map $\sigma_E$ on the boundary path space $\partial E$.
        Then $\mathcal{O}(E) \cong C^*_r(\G(\partial E,\sigma_E))$.
	\end{prop}
	
	\subsection{Minimal dynamics on point-like spaces}\label{subsec:point-like}
	
	A homeomorphism $\rho:X \rightarrow X$ from a topological space $X$ to itself is said to be \emph{minimal} if the only closed $\rho$-invariant subsets of $X$ are $\emptyset$ and $X$. 
	
	In general, there are topological obstructions to the existence of a minimal homeomorphism on a given space (see \cite{deeley2023min-homeo} for a discussion). In particular, the Lefschetz fixed point theorem (\cite{Lefschetz37}) implies that no even-dimensional sphere can be equipped with a minimal homeomorphism. However, minimal homeomorphisms are known to exist on all odd-dimensional spheres (\cite{FH77}). 
	
	Motivated by the goal of finding a dynamical model for the Jiang--Su algebra $\mathcal{Z}$, Deeley, Putnam and Strung constructed a minimal homeomorphism on a space $Z$ that has the same cohomology and $K$-theory as a point (\cite{deeley2018constructing}). The existence of such a minimal dynamical system is consistent with the Lefschetz fixed point theorem because the space $Z$ does not satisfy the regularity conditions of the fixed point theorem.
	
	Loosely speaking, the space $Z$ is constructed by starting with a minimal homeomorphism $\rho:S^d \rightarrow S^d$ where $d$ is odd and $d \geq 3$.
	A suitable $\rho$-invariant subset $L_\infty \subseteq S^d$ is then removed and the resulting space $S^d \setminus L_\infty$ is completed with respect to a suitable metric. The map $\rho$ is then shown to induce a minimal homeomorphism on $Z$, and the desired topological properties of $Z$ are established by realising $Z$ as an inverse limit of contractible spaces. 
	
	The following theorem summarises the key properties of the Deeley--Putnam--Strung construction that we will use in this paper. We write $\{*\}$ for the topological space with a single point. 
	
	\begin{thm}[{\cite{deeley2018constructing}  and \cite[Theorem 2.5]{deeley2023min-homeo}}] \label{thm:DPS-Theorem}
		Let $d$ be an odd number with $d \geq 3$. There exists an infinite compact metric space $Z$ with covering dimension $d$ or $d-1$ and a minimal homeomorphism $\rho:Z \rightarrow Z$ such that, for any continuous generalised cohomology theory $H^*(\cdot)$, there is an isomorphism of groups $H^*(Z) \cong H^*(\{*\})$. In particular, this holds for \v{C}ech cohomology and $K$-theory.
	\end{thm}
	
	Translating from topological $K$-theory to C$^*$-algebraic $K$-theory, we observe that the C$^*$-algebra $C(Z)$ of continuous functions on the space $Z$ from Theorem \ref{thm:DPS-Theorem} has the same $K$-theory as the complex numbers $\C$. In fact, $C(Z)$ is KK-equivalent to $\C$ by \cite[Corollary 1.12]{deeley2018constructing}.

	\section{Main results}
	
	In this section, we prove the main results stated in the introduction.
	We shall begin with the proof of Theorem \ref{thm:general}, which is the key construction of the paper.
	The main idea is to combine Katsura's topological graph construction from \cite[Proposition 2.9]{katsura2008class} with the minimal dynamical system of Deeley--Putnam--Strung discussed in Section~\ref{subsec:point-like} above.
	
	\begin{thm}[Theorem~\ref{thm:general}]
		Let $X$ be a second-countable, locally compact, Hausdorff space. Then there exists a principal, second-countable, locally compact, Hausdorff, étale groupoid $\G$ such that $C_r^*(\G)$ is a UCT Kirchberg algebra satisfying
		\begin{align}
			K_0(C_r^*(\G)) &\cong K_0(C_0(X)),\label{eqn:ThmB-K0}\\
			K_1(C_r^*(\G)) &\cong K_1(C_0(X)).
		\end{align}
		If $X$ has finite covering dimension, then the unit space $\G^{(0)}$ has finite covering dimension.
		If $X$ is compact, then the unit space $\G^{(0)}$ is compact, the C$^*$-algebra $C_r^*(\G)$ is unital, and the isomorphism of $K_0$-groups can be chosen to preserve the $K_0$-class of the unit. 
	\end{thm}
	\begin{proof}
		Fix a second-countable, locally compact, Hausdorff space $X$. Since $X$ is second-countable, we can choose a sequence $(x_i)_{i=1}^\infty$ in $X$ such that every non-empty open subset of $X$ contains infinitely many elements of $(x_i)_{i=1}^\infty$.
		Let $Z$ denote the infinite compact metric space coming from Theorem \ref{thm:DPS-Theorem} (with $d=3$). Then $C(Z)$ has the same $K$-theory as $\C$, there exists a minimal homeomorphism $\rho$ on $Z$, and $\dim(Z) \in \{2,3\}$. 
		
		Consider the topological graph $E= (E^0, E^1,d,r)$ with vertex set $E^0 = Z \times X$, edge set $E^1 = Z \times X \times \N$, and domain and range maps given by $d(z,x,m) = (z,x)$ and $r(z,x,m) = (\rho(z),x_m)$ for all $(z,x,m) \in E^1$.
		Note that Definition \ref{def:top-graph} is satisfied as $d$ is a local homeomorphism and $r$ is continuous. 
		We will show that $\mathcal{O}(E)$ is a stable UCT Kirchberg algebra with the same $K$-theory as $C_0(X)$. 
		
		We first observe that $E$ is second-countable because both $X$ and $Z$ are second-countable.
		Next, we prove that $E$ is minimal. 
		Let $(z,x) \in E^0$. Then 
		\begin{equation}
			\text{Orb}^+((z,x)) = \{(z,x)\}\cup \{(\rho^n(z),x_m) : n,m\in \N \}.
		\end{equation}
		By minimality of $\rho$ and the choice of $(x_i)_{i=1}^\infty$, it follows that $\text{Orb}^+((z,x))$ is dense in $E^0$. 
		For any negative orbit $\mu$ of $(z,x)$, $\text{Orb}^+((z,x)) \subseteq \text{Orb}((z,x),\mu)$. Hence, $\text{Orb}((z,x),\mu)$ is dense in $E^0$. Therefore, $E$ is minimal.
		
		We now prove that $E$ is contracting. Fix $z \in Z$. We shall show that $E$ is contracting at the point $(z,x_1) \in E^0$.  As in the proof of minimality above, the forward orbit of $(z,x_1)$ is dense in $E^0$, so $\overline{\text{Orb}^+((z,x_1))} = E^0$.
		Let $W \subseteq E^0$ be a neighbourhood of $(z,x_1)$. Since $Z$ and $X$ are locally compact, Hausdorff spaces, there exist relatively compact open subsets $U \subseteq Z$ and $V \subseteq X$ such that $(z,x_1)\in U \times V \subseteq \overline{U} \times \overline{V} \subsetneq W$. Note that $\overline{U \times V } = \overline{U} \times \overline{V}$  is compact and a proper subset of $Z \times X$. 
		
		The set $\bigcup^\infty_{k=1} \rho^{-(k+1)}(U)$ is non-empty, open and $\rho^{-1}$-invariant. As $\rho$ is minimal, so is $\rho^{-1}$. Hence, $\bigcup^\infty_{k=1} \rho^{-(k+1)}(U) = Z$. By compactness of $Z$, there is $N \in \N$ such that
		\begin{equation} \label{eqn:translates-of-U}
			\bigcup^N_{k=1} \rho^{-(k+1)}(U) = Z.
		\end{equation}
		For $x \in X$, $z \in U$ and $k \in \{1, \dots, N\}$, let $e_{x,z,k} \in E^{k+1}$ be the path
		\begin{equation} \label{eqn:path-def}
			((\rho^{-1}(z),x_k,1),(\rho^{-2}(z),x_k,k),  \cdots,  (\rho^{-k}(z),x_k,k),(\rho^{-(k+1)}(z),x,k)).
		\end{equation}
		Then $r(e_{x,z,k}) = r(\rho^{-1}(z),x_k,1) = (z,x_1)$ and $d(e_{x,z,k}) = d(\rho^{-(k+1)}(z),x,k) = (\rho^{-(k+1)}(z),x)$. 
		Set $U_k = \{e_{x,z,k}:  x \in X, z \in U\} \subseteq E^{k+1}$ for $k \in \{1, \dots , N\}$. 
		Note that, for each $k \in \{1, \dots , N\}$, $U_k$ is open in $E^{k+1}$ as it coincides with the intersection of $E^{k+1}$ and the open set 
		\begin{equation} 
			(\rho^{-1}(U) \times X \times \{1\}) \times \left(\prod_{i=2}^{k} \rho^{-i}(U) \times X  \times \{k\}\right)  \times  (\rho^{-(k+1)}(U) \times X \times \{k\}).
		\end{equation}
		We first compute that $r(U_k) = U \times\{ x_1 \}\subseteq U \times V$ for all $k \in \{1, \dots , N\}$. We next observe that $U_k \pitchfork U_l = \emptyset$ holds for all $k, l \in \{1, \dots , N\}$ with $k \neq l$ by \eqref{eqn:path-def}.
		Finally, by \eqref{eqn:translates-of-U}, we have 
		\begin{equation}
			\overline{U \times V } \subsetneq Z \times X =  \bigcup_{k=1}^N \rho^{-(k+1)}(U) \times X = \bigcup_{k=1}^N d(U_k).
		\end{equation}
		Hence, $U \times V$ is a contracting open set. Therefore, $E$ is a contracting topological graph.
		Since $E$ is second-countable, minimal and contracting, the topological graph C$^*$-algebra $\mathcal{O}(E)$ is a UCT Kirchberg algebra by Theorem \ref{thm:graph-contracting-Kirchberg-algebra}. 
		
		We now compute the $K$-theory of $\mathcal{O}(E)$ using Proposition~\ref{prop:graph-k-theory}.
		Note that the choice of $(x_i)_{i=1}^\infty$ ensures that $r^{-1}(V)$ is non-compact for every non-empty open subset $V \subseteq E^0$ as the projection of $r^{-1}(V)$ onto the third coordinate is an infinite subset of $\N$. Hence, $E^0_{rg}= \emptyset$.
		
		Applying Proposition \ref{prop:graph-k-theory}, the Künneth formula (\cite{schochet1982topological}) and Theorem \ref{thm:DPS-Theorem}, we obtain
		\begin{align}
			K_0(\mathcal{O}(E))  &\cong K_0(C_0(E^0)) = K_0(C_0(Z \times X)) \cong K_0(C_0( X)),\label{eqn:K0-iso}\\
			K_1(\mathcal{O}(E))  &\cong K_1(C_0(E^0)) = K_1(C_0(Z \times X)) \cong K_1(C_0( X)).\label{eqn:K1-iso}
		\end{align}
		
		Let us now turn towards the groupoid model. Let $\G = \G(\partial E, \sigma_E)$ be the Deaconu--Renault groupoid corresponding to the shift map $\sigma_E$ on the boundary path space $\partial E$ of $E$.
		By Proposition \ref{prop:groupoid-model-for-graphs}, $C^*_r(\G) \cong \mathcal{O}(E)$.
		As $E$ is second-countable, $\partial E$ is a second-countable, locally compact, Hausdorff space. Therefore, the Deaconu--Renault groupoid $\G$ is second-countable, locally compact, Hausdorff and étale. Moreover, we may identify the unit space of $\G$ with $\partial E$ (see Section \ref{subsec:DR-groupoids}).
		
		We now show that $\G$ is principal.
		Fix $\mu \in \G^{(0)} = \partial E$. Let $g \in \G$ with $r(g) = s(g) = \mu$.
		Then there exist $k,l \in \N_0$ with $\max(k,l) \leq |\mu|$ such that $\sigma_E^k(\mu) = \sigma_E^l(\mu)$ and $g=(\mu,k-l,\mu)$.
		In particular, we have $|\sigma_E^k(\mu)| = |\sigma_E^l(\mu)|$.
		If $\mu$ is a finite path, this immediately implies $k = l$. Hence, $g \in \G^{(0)}$ in this case.
		
		Suppose now that $\mu \in E^\infty$. For $\mu$ to be a valid infinite path in $E$, we must have 
		$\mu = ((\rho^{-i}(z),x_{n_{i+1}},n_i))_{i=1}^\infty$ for some $z \in Z$ and some sequence of natural numbers $(n_i)_{i=1}^\infty$. 
		Since $\sigma_E^k(\mu) = \sigma_E^l(\mu)$, it follows that $\rho^{-(k+1)}(z) = \rho^{-(l+1)}(z)$. Since $\rho$ is a minimal homeomorphism on an infinite space, the induced $\Z$-action is free. Therefore, we must have $k=l$, and so $g \in \G^{(0)}$ in this case as well. This completes the proof that $\G$ is principal. 
		
		Suppose $X$ is compact. Then $E^0 = Z \times X$ is compact. By \cite[Proposition 7.1]{katsura2006class2}, $\mathcal{O}(E)$ is unital. Hence, $C^*_r(\G) \cong \mathcal{O}(E)$ is unital, from which it follows that $\G^{(0)}$ is compact.
		Moreover, by \cite[Proposition 7.1]{katsura2006class2}, the $^*$-homomorphism $t^0:C_0(E^0) \rightarrow \mathcal{O}(E)$ is unital. Hence, the isomorphism $t^0_*:K_0(C_0(E^0)) \rightarrow K_0(\mathcal{O}(E))$ given by Proposition \ref{prop:graph-k-theory} preserves the $K_0$-class of the unit. As the isomorphism $K_0(C_0(X)) \cong K_0(C_0(Z \times X))$ is induced by tensoring with $[1_{C(Z)}]_{0}$, the isomorphism \eqref{eqn:K0-iso} preserves the $K_0$-class of the unit. Therefore, the isomorphism \eqref{eqn:ThmB-K0} preserves the $K_0$-class of the unit.

		Suppose $X$ has finite covering dimension. We shall prove that $\G^{(0)}$ has finite covering dimension. 
		As $E$ is assumed to be second-countable, $\partial E$ is metrisable and its topology is completely determined by Lemma~\ref{lem:Kumjian-Li}, which describes the convergence of sequences.
		
		As every vertex in $E$ is singular, we have $\partial E = E^\infty \cup E^*$. It is important to note that, although this is a disjoint union at the level of sets, neither $E^\infty$ nor $E^*$ will be closed in $\partial E$ in general.  
		However, it follows from Lemma~\ref{lem:Kumjian-Li} that $\bigcup_{k=0}^n E^k \subseteq \partial E$, the set of paths of length at most $n$, is closed in $\partial E$ for each $n \in \N$.
		
		To get an upper bound on the covering dimension of $\partial E$, we use \cite[Theorem 3.1.17]{engelking1978dimension}, which gives
		\begin{equation}\label{eqn:sledgehammer}
			\dim(\partial E) \leq \dim(E^\infty) + \dim(E^*)+1.
		\end{equation}
		We note that our use of \cite[Theorem 3.1.17]{engelking1978dimension} is valid since $\partial E$ is metrisable and thus hereditarily normal.
		
		As mentioned above, the infinite paths in $E$ are parameterised by the map $f: Z \times \N^\N \to E^\infty$ given by 
		\begin{equation}
			(z,(n_i)_{i=1}^\infty) \mapsto ((\rho^{-i}(z),x_{n_{i+1}},n_i))_{i=1}^\infty.
		\end{equation}
		It follows from Lemma~\ref{lem:Kumjian-Li} that $f$ is a homeomorphism. Since $\N^\N$ is zero-dimensional, we have $\dim(E^\infty) \leq \dim(Z) + 0 = \dim(Z)$ by \cite[Theorem 3.2.14]{engelking1978dimension} (which is applicable as second-countable metrisable spaces are normal and strongly paracompact.)
		
		Let $k \in \N$ with $k \geq 2$. The paths in $E$ of length $k$ are parameterised by the function $f_k: Z \times X \times \N^k \to E^k$ that maps $(z,x,(n_i)_{i=1}^k)$ to the path
		\begin{equation}
			((\rho^{-1}(z),x_{n_2},n_1),\cdots, (\rho^{-(k-1)}(z),x_{n_k},n_{k-1}),(\rho^{-k}(z),x,n_k)).
		\end{equation}
		It follows from Lemma~\ref{lem:Kumjian-Li} that $f_k$ is a homeomorphism. 
		In fact, $E^k \cong Z \times X \times \N^k$ holds for all $k \in \N_0$ since in the case $k \in \{0,1\}$ we can just use the definition of $E^1$ and $E^0$. By \cite[Theorem 3.2.14]{engelking1978dimension}, $\dim(E^k) \leq \dim(Z) + \dim(X) + 0$.
		As $\bigcup_{k=0}^n E^k \subseteq E^*$ is closed for all $n \in \N$, we may use \cite[Proposition 3.1.7]{engelking1978dimension} to obtain $\dim(E^*) \le \dim(Z) + \dim(X)$.
		
		Substituting our upper bounds for $\dim(E^\infty)$ and $\dim(E^*)$ into \eqref{eqn:sledgehammer} gives 
		\begin{equation}
			\dim(\partial E) \leq \dim(Z)+(\dim(Z) + \dim(X))+1 < \infty. 
		\end{equation}
		This completes the proof that $\G^{(0)}$ has finite covering dimension.
	\end{proof}
	
	As explained in the introduction, we can deduce Theorem \ref{thm:stableKirchberg} from Theorem \ref{thm:general} using the Kirchberg--Phillips theorem (\cite{Ki95,Ph00}).
	
	\begin{thm}[Theorem~\ref{thm:stableKirchberg}]
		Let $A$ be a stable Kirchberg algebra satisfying the UCT. Then there exists a principal, second-countable, locally compact, Hausdorff, étale groupoid $\G$ with a finite-dimensional unit space such that $A \cong C_r^*(\G)$. 
	\end{thm}
	\begin{proof}
		Set $G_0 = K_0(A)$ and $G_1 = K_1(A)$.
		By \cite[Corollary 23.10.3]{Bla98}, there exists a second-countable, locally compact, Hausdorff space $X$ with covering dimension at most three such that $K_0(C_0(X)) \cong G_0$ and  $K_1(C_0(X)) \cong G_1$.
		
		Let $\G$ be the principal groupoid obtained from Theorem~\ref{thm:general} for this choice of starting space $X$. Then $C_r^*(\G)$ is a UCT Kirchberg algebra with the same $K$-theory as $C_0(X)$. As $X$ is finite-dimensional, the unit space $\G^{(0)}$ has finite covering dimension.
		
		Let $\mathcal{R}_\N$ be the complete equivalence relation on $\N$ with the discrete topology. Then $C_r^*(\mathcal{R}_\N)$ is isomorphic to the compact operators $\K(\ell^2)$. The product groupoid $\G \times \mathcal{R}_\N$ is principal as both $\G$ and $\mathcal{R}_\N$ are principal, and $(\G \times \mathcal{R}_\N)^{(0)} \cong  \G^{(0)} \times \N$ has the same covering dimension as $\G^{(0)}$. By \cite[Lemma 5.1]{barlak2017cartan},
		$C^*_r(\G \times \mathcal{R}_\N) \cong C^*_r(\G) \otimes \K(\ell^2)$.
		
		By replacing $\G$ with $\G \times \mathcal{R}_\N$, we may assume without loss of generality that $C_r^*(\G)$ is a stable UCT Kirchberg algebra with the same $K$-theory as $C_0(X)$. 
		By the stable version of the Kirchberg--Phillips theorem, $A \cong C_r^*(\G)$; see for example \cite[Theorem 8.4.1(ii)]{rordam2002classification}. 
	\end{proof}
	
	In the case of unital UCT Kirchberg algebras, it is important to keep track of the position of the class of the unit in the  $K_0$-group. Note that, if $(K_0(A), [1_A]_0) \cong (K_0(C(X)), [1_{C(X)}]_0)$ for some compact space $X$, then $[1_A]_0$ has infinite order and generates a direct summand of $K_0(A)$. Indeed, any evaluation map $C(X) \rightarrow \C$ is a left inverse to $\C 1_{C(X)} \hookrightarrow C(X)$, which forces $K_0(\C 1_{C(X)})$ to be a direct summand of $K_0(C(X))$. It turns out that this is the only restriction we require to obtain the following unital version of Theorem~\ref{thm:stableKirchberg} using the construction of Theorem~\ref{thm:general}.
	
	\begin{thm}[Theorem~\ref{thm:MostUnitalKirchberg}]
		Let $A$ be a unital Kirchberg algebra satisfying the UCT. Suppose $[1_A]_0$ has infinite order and generates a direct summand of $K_0(A)$. Then there exists a principal, second-countable, locally compact, Hausdorff, étale groupoid $\G$ with a finite-dimensional unit space such that $A \cong C_r^*(\G)$. In particular, $A$ contains a C$^*$-diagonal.
	\end{thm}
	\begin{proof}
		Set $G_1 = K_1(A)$. By hypothesis, $K_0(A) \cong \Z[1_A]_0 \oplus G_0$ for some subgroup $G_0$.  
		By \cite[Corollary 23.10.3]{Bla98}, there exists a second-countable, locally compact, Hausdorff space $Y$ with covering dimension at most three such that $K_0(C_0(Y)) \cong G_0$ and  $K_1(C_0(Y)) \cong G_1$.
		Let $X$ be the one-point compactification of $Y$. Then $C(X)$ is isomorphic to the unitisation of $C_0(Y)$. (Our convention when $Y$ is already compact is that $X \cong Y \sqcup \{*\}$ and $C(X) \cong C(Y) \oplus \C$.) 
		It follows that $K_1(C(X)) \cong K_1(C_0(Y)) \cong G_1$ and
		$K_0(C(X)) = \Z[1_{C(X)}]_0 \oplus K_0(C_0(Y)) \cong \Z[1_A]_0 \oplus G_0$,
		where the isomorphism maps $[1_{C(X)}]_0$ to $[1_A]_0$. Moreover, $\dim(X) \leq \dim(Y) + 1 < \infty$ by \cite[Theorem 3.1.17]{engelking1978dimension}.
		
		Let $\G$ be the groupoid obtained from Theorem~\ref{thm:general} for this choice of starting space $X$. Then $C_r^*(\G)$ is a unital UCT Kirchberg algebra whose $K$-theory is isomorphic to that of $C(X)$ via a map that respects the class of the unit in $K_0$. Moreover, the unit space of $\G$ is finite-dimensional.
		By the unital version of the Kirchberg--Phillips theorem, $A \cong C_r^*(\G)$; see for example \cite[Theorem 8.4.1(iv)]{rordam2002classification}. 
		By Theorem~\ref{thm:renault-cartan}, $C(\G^{(0)})$ gives rise to a C$^*$-diagonal in $A$. 
	\end{proof}
	
	We now turn to the special case of the Cuntz algebra $\O_\infty$. 
	\begin{thm}[Theorem~\ref{thm:O_infinity}]
		There exists a principal, second-countable, locally compact, Hausdorff, étale groupoid $\G_\infty$ such that 
        $C_r^*(\G_\infty) \cong \O_\infty$. 
		In particular, $\O_\infty$ contains a C$^*$-diagonal. Moreover, the unit space of $\G_\infty$ is the product of a Cantor space and a Deeley--Putnam--Strung space and has dimension at most three.
	\end{thm}
	\begin{proof}
		As $K_0(\O_\infty) \cong \Z$ with generator $[1_{\O_\infty}]_0$, Theorem~\ref{thm:MostUnitalKirchberg} applies. In fact, since $K_1(\O_\infty) = 0$, by reviewing the proof of Theorem~\ref{thm:MostUnitalKirchberg}, we see that the groupoid $\G_\infty$ is obtained by applying the construction of Theorem~\ref{thm:general} to the one-point space  $X = \{*\}$. In this case, we can identify the unit space of the groupoid arising in the construction.
		
		Let $E$ be the topological graph arising in the proof of Theorem~\ref{thm:general}, and let $F=(\{*\},\N, n \mapsto *, n \mapsto *)$ be the standard graph algebra model for $\O_\infty$. Then $\partial E$ is homeomorphic to $Z \times \partial F$, where $Z$ is the Deeley--Putnam--Strung space used in the construction of $E$. Indeed, it follows from Lemma~\ref{lem:Kumjian-Li} that the bijection $h: Z \times \partial F \rightarrow \partial E$ given by the identity map on $Z \times F^0$, and on non-trivial paths (both finite and infinite) by $h(z,(m_i)_{i=1}^k) = ((\rho^{-i}(z),*,m_i))_{i=1}^k$,
		where $k \in \N \cup \{\infty\}$, $m_i \in \N$ and $z \in Z$, is a homeomorphism. 
		
		The space $\partial F$ is well-known to be a Cantor space (see for example \cite[Section~3.1]{brown2016purely}), so $\partial E$ is the product of a Cantor space $\partial F$ and the Deeley--Putnam--Strung space $Z$. As $\dim(\partial F) = 0$, we have $\dim(Z) \leq \dim(\partial E) \leq \dim(Z)$ by \cite[Theorems 3.1.23 and 3.2.14]{engelking1978dimension}, so $\dim(\partial E) = \dim(Z) \in \{2,3\}$. By Theorem~\ref{thm:renault-cartan}, $C(\G^{(0)}_\infty)$ gives rise to a C$^*$-diagonal in $\O_\infty$. 
	\end{proof}
	
	Finally, we come to the proof of Theorem~\ref{thm:EvenMoreUnitalKirchberg}, which combines Theorem~\ref{thm:O_infinity} with results of Xin Li in the stably finite setting (\cite{li2020every}).
	\begin{thm}[Theorem~\ref{thm:EvenMoreUnitalKirchberg}]
		Let $A$ be a unital Kirchberg algebra satisfying the UCT. Suppose $[1_A]_0$ has infinite order in $K_0(A)$. Then there exists a principal, second-countable, locally compact, Hausdorff, étale groupoid $\G$ with a finite-dimensional unit space and a twist $\Sigma$ such that $A \cong C_r^*(\G, \Sigma)$. In particular, $A$ contains a C$^*$-diagonal.
	\end{thm}
	\begin{proof}
		Let $G_0 = K_0(A)$ and $u = [1_A]_0 \in G_0$. Since $u$ has infinite order, there exists a group homomorphism $\phi:G_0 \rightarrow \R$ satisfying $\phi(u) = 1$; for example, map $G_0$ into the $\R$-vector space $G_0 \otimes_{\Z} \R$, and use a basis to construct the required linear functional on $G_0 \otimes_{\Z} \R$. Set $G_0^+ = \{g \in G_0: \phi(g) > 0\} \cup \{0\}$. Then $(G_0,G_0^+,u)$ is a weakly unperforated, simple, scale-ordered, countable, abelian group. Let $G_1 = K_1(A)$. Then $G_1$ is a countable abelian group. Let $T = \{*\}$ be the one-point simplex and let $r:T \rightarrow S(G_0)$ be the pairing map given by $r(*) = \phi$. Then $(G_0,G_0^+,u, G_1,T,r)$ is a valid Elliott invariant.
		
		By the inductive limit construction of \cite{li2020every} (which has its origins in \cite{El96}), there exists a stably finite, simple, separable, amenable, unital, $\mathcal{Z}$-stable C$^*$-algebra $B$ satisfying the UCT such that $(K_0(B), K_0(B)^+, [1_B]_0) \cong (G_0, G_0^+,u)$, $K_1(B) \cong G_1$ and $T(B) \cong T$ (and the pairing between $K_0(B)$ and $T(B)$ induces $r$).
		By \cite[Theorem 1.2]{li2020every}, there exists a principal, second-countable, locally compact, Hausdorff, étale groupoid $\G_{B}$ and a twist $\Sigma_B$ such that $B \cong C_r^*(\G_B, \Sigma_B)$. Moreover, the unit space of $\G_{B}$ has finite dimension by \cite[Corollary 1.8]{li2020every}. 
		
		Let $\G_\infty$ be the principal groupoid model for $\O_\infty$ coming from Theorem~\ref{thm:O_infinity}. Set $\G = \G_B \times \G_\infty$. Let $\Sigma$ be the product of the twist $\Sigma_B$ and the trivial twist $\Sigma_\infty = \mathbb{T} \times \G_\infty$ as defined in \cite[Section 5]{barlak2017cartan}. Then $C_r^*(\G, \Sigma) \cong  C_r^*(\G_B, \Sigma_B) \otimes C_r^*(\G_\infty) \cong B \otimes \O_\infty$ by \cite[Lemma 5.1]{barlak2017cartan} and Theorem~\ref{thm:O_infinity}. Since $B$ is simple, separable, amenable, unital and satisfies the UCT, the tensor product $B \otimes \O_\infty$ is a unital UCT Kirchberg algebra. 
		Moreover, tensoring with $\O_\infty$ preserves $K$-theory, including the position of the class of the unit in $K_0$, thanks to the Künneth formula.
		By the unital version of the Kirchberg--Phillips theorem, $A \cong B \otimes \O_\infty \cong C_r^*(\G, \Sigma)$.
		By Theorem~\ref{thm:renault-cartan}, $C(\G^{(0)})$ gives rise to a C$^*$-diagonal in $A$. 
	\end{proof}

	We end the paper with a couple of further remarks about the unital case.
	\begin{rmk}
		The twist in Theorem~\ref{thm:EvenMoreUnitalKirchberg} will be trivial whenever there is an untwisted principal groupoid model for the stably finite classifiable C$^*$-algebra $B$ with the same $K$-theory. In particular, this holds whenever $K_0(A)$ is torsion-free by \cite[Corollary 1.8(i)]{li2020every}.
	\end{rmk}
	
	\begin{rmk}
		Every unital UCT Kirchberg algebra $A$ is isomorphic to the corner $(1_A \otimes e_{11})(A \otimes \K(\ell^2))(1_A \otimes e_{11})$ of its stabilisation. One can apply Theorem \ref{thm:stableKirchberg} to obtain a principal groupoid model $\G$ for $A \otimes \K(\ell^2)$. If the $K_0$-class of $1_A \otimes e_{11}$ can be realised by a clopen subset $V$ of the unit space of $\G$, then the reduction of the groupoid $\G$ to the clopen subset $V$ will be a principal groupoid model for $A$. However, the construction of Theorem~\ref{thm:general} does not guarantee that all elements of $K_0(\G)$ are realised in this way, since we are identifying $K_0(\G)$ with the $K$-theory of a topological space.  
	\end{rmk}

\end{document}